\documentclass[11pt,hidelinks]{article}
\usepackage{orcidlink}

\title{Gaussian deconvolution on $\mathbb R^d$ \\
with application to self-repellent Brownian motion}

\author{
Yucheng Liu\,\orcidlink{0000-0002-1917-8330}\thanks{Department of Mathematics,
	University of British Columbia,
	Vancouver, BC, Canada V6T 1Z2.
	\href{mailto:yliu135@math.ubc.ca}{yliu135@math.ubc.ca}.
	}
}
\date{\vspace{-5ex}} 

\usepackage{amsmath, amssymb, amscd, amsthm, amsfonts}
\usepackage{graphicx} 
\usepackage{hyperref} 
\usepackage{booktabs}
\usepackage{xcolor}
\usepackage{ dsfont, bm}
\usepackage[title]{appendix}
\usepackage{comment}
\usepackage{enumerate}
\usepackage{enumitem}
\usepackage{cite}

\usepackage[textwidth=480pt,textheight=650pt,centering]{geometry} 

\theoremstyle{plain}
\newtheorem{theorem}{Theorem}[section]
\newtheorem{lemma}[theorem]{Lemma}

\newtheorem{proposition}[theorem]{Proposition}

\newtheorem{assumption}[theorem]{Assumption}
\newtheorem{remark}[theorem]{Remark}

\numberwithin{equation}{section}

\newcommand{\ie}{i.e.}
\newcommand{\eg}{e.g.}

\newcommand{\eps}{\varepsilon}

\newcommand{\N}{\mathbb{N}}
\newcommand{\Z}{\mathbb{Z}}
\newcommand{\Q}{\mathbb{Q}}
\newcommand{\R}{\mathbb{R}}
\newcommand{\C}{\mathbb{C}}

\renewcommand{\P}{\mathbb{P}}
\newcommand{\T}{\mathbb{T}}
\providecommand{\Rd}{{\mathbb R^d}}
\providecommand{\Zd}{{\mathbb Z^d}}

\newcommand{\Ccal}{\mathcal{C}}

\newcommand{\grad}{\nabla}
\newcommand{\inv}{^{-1}}
\renewcommand{\(}{\left(}
\renewcommand{\)}{\right)}

\newcommand{\1}{\mathds{1}}

\newcommand{\nl}{\nonumber \\}
\newcommand{\loc}{_{\mathrm{loc}}}

\providecommand{\abs}[1]{\lvert#1\rvert}

\providecommand{\biggabs}[1]{\biggl\lvert#1\biggr\rvert}

\providecommand{\norm}[1]{\lVert#1\rVert}
\providecommand{\bignorm}[1]{\bigl\lVert#1\bigr\rVert}

\usepackage{mathrsfs}

\providecommand{\floor}[1]{\lfloor #1 \rfloor}

\newcommand{\crit}{_{\lambda_c}}
\newcommand{\KIR}{K_{\rm IR}}
\newcommand{\B}{{B_1}}

\begin{document}
\maketitle

\begin{abstract}
We consider the convolution equation $(\delta - J) * G = g$ on $\mathbb R^d$, $d>2$, where $\delta$ is the Dirac delta function and $J,g$ are given functions. 
We provide conditions on $J, g$ that ensure the deconvolution $G(x)$ to decay as $( x \cdot \Sigma\inv x)^{-(d-2)/2}$ for large $\abs x$, where $\Sigma$ is a positive-definite diagonal matrix. 
This extends a recent deconvolution theorem on $\mathbb Z^d$ proved by the author and Slade to the possibly anisotropic, continuum setting while maintaining its simplicity. 
Our motivation comes from studies of statistical mechanical models on $\mathbb R^d$ based on the lace expansion. 
As an example, we apply our theorem to a self-repellent Brownian motion in dimensions $d>4$, proving its critical two-point function to decay as $\abs x^{-(d-2)}$, like the Green function of the Laplace operator $\Delta$. 
\end{abstract}

%


\section{Introduction and results}
\subsection{Introduction}

Deconvolution theorems have been very useful in studies of statistical mechanical models on the integer lattice $\Zd$, above the upper critical dimension. 
In the 2000s, using convolution equations provided by the \emph{lace expansion}, 
Hara, van der Hofstad, Slade \cite{HHS03} and Hara \cite{Hara08}
proved deconvolution theorems that established $\abs x^{-(d-2)}$ decay of the critical two-point functions for the self-avoiding walk, Bernoulli bond percolation, and lattice trees and lattice animals. 
Recently, inspired by \cite{Slad22_lace},
a much simpler deconvolution theorem that yields the same results was proved by the author and Slade \cite{LS24a, LS24b}, using only elementary Fourier analysis and H\"older's inequality. 
The $\abs x^{-(d-2)}$ decay of the critical two-point function is useful, \eg, in percolation theory to study arm-exponents \cite{CHS23,KN11} and the incipient infinite cluster \cite{HJ04,KN09}.

The lace expansion was originally developed to study the self-avoiding walk in dimensions $d>4$ \cite{BS85,HS92a, BHK18, Slad22_lace}, 
and the method has been extended to many models on $\Zd$, 
including percolation in $d>6$ \cite{HS90a, HH17book, FH17}, 
Ising and $\varphi^4$ models in $d>4$ \cite{BHH21, Saka07, Saka15}, 
and lattice trees and lattice animals in $d>8$ \cite{HS90b,FH21}.
More recently, 
lace expansions have also been derived for statistical mechanical models on $\Rd$. 
These include the random connection model in $d>6$ \cite{HHLM19} and the self-repellent Brownian motion in $d > 4$ \cite{BKM24}.

Motivated by lace expansion equations on $\Rd$, we study convolution equations. 
Our main result is a simple deconvolution theorem on $\Rd$ similar to that of \cite{LS24a}, with several extensions. 
Unlike \cite{LS24a} which uses $\Zd$-symmetry of the functions, 
we use only the even symmetry. This generalisation allows us to obtain \emph{anisotropic} $\abs x^{-(d-2)}$ decay in the solution, which is more natural for models on $\Rd$ which might not possess symmetries of $\Zd$. 
We also formulate our hypotheses in terms of moments of the functions, which are weaker and often easier to verify than the decay hypotheses of \cite{LS24a, Hara08}. 
Regarding the method of proof, 
we maintain the simplicity of \cite{LS24a} by using the Fourier transform and weak derivatives in the Fourier space. 
However, since the Fourier dual of $\Rd$ is the non-compact $\Rd$, 
there are additional difficulties arising from the far-field behaviour of the Fourier transform, and we provide the extra control needed. 

As an example application, we apply our new deconvolution theorem to the self-repellent Brownian motion in dimensions $d>4$, and we prove that its critical two-point function is asymptotic to $\abs x^{-(d-2)}$ when $\abs x$ is large, improving the upper bound obtained in \cite{BKM24}.
In another work in preparation, 
we apply our result to the random connection model in dimensions $d>6$. This is a model where anisotropic $\abs x^{-(d-2)}$ decay of the critical connection probability is possible, by choosing a connection function that is not $\Zd$-symmetric.

\medskip\noindent
{\bf Notation.}  
We write
$a \vee b = \max \{ a , b \}$ and $a \wedge b = \min \{ a , b \}$.
We write
$f = O(g)$ or $f \lesssim g$ to mean there exists a constant $C> 0$ such that $\abs {f(x)} \le C\abs {g(x)}$, and $f= o(g)$ for $\lim f/g = 0$.

Given a multi-index $\alpha = (\alpha_1, \dots, \alpha_d) \in \N_0^d$ and $x\in \Rd$, 
we write $\abs \alpha = \sum_{i=1}^d \alpha_j$, $x^\alpha = \prod_{i=1}^d x_i^{\alpha_i}$, and $\grad^\alpha = \prod_{i=1}^d \grad_i^{\alpha_i}$.

\medskip\noindent
{\bf Fourier transform.}  
For functions $f, \hat g \in L^1(\Rd)$, the Fourier transform and its inverse are given by
\begin{equation}
\hat f(k) = \int_{\R^d} f(x) e^{ik\cdot x} dx	
	\quad (k \in \R^d),
\qquad
g(x) = \int_{\R^d} \hat g(k) e^{-ik\cdot x}  \frac{ dk }{ (2\pi)^d } 
	\quad (x\in \R^d) .
\end{equation}

\subsection{Gaussian deconvolution theorem}

Let $d > 2$. 
Let $(f*h)(x) = \int_\Rd f(y)h(x-y) dy$ denote the convolution of $L^1(\Rd)$ functions, and let $\delta$ denote the convolution identity (the Dirac delta function). We consider the convolution equation
\begin{equation} \label{eq:JG}
(\delta - J) * G = g ,
\end{equation}
where $J, g: \Rd \to \R$ are given functions that obey certain regularity assumptions. 
Our deconvolution theorem will show that the solution $G(x)$ to \eqref{eq:JG} decays as 
$( x \cdot \Sigma\inv x)^{-(d-2)/2}$ for large $\abs x$, 
where $\Sigma$ is a $J$-dependent diagonal matrix.
Apart from the possible anisotropy, this decay is
like the Green function $G_\Delta(x) = c_d \abs x^{-(d-2)}$ of the Laplace operator $\Delta$. 
In fact, 
if we take $g = J$ in \eqref{eq:JG}, then $G$ is the Green function of the operator $f \mapsto (\delta - J) * f$ up to a Dirac delta: Formally, $\delta + G$ satisfies the equation
\begin{equation}
(\delta - J) * ( \delta + G ) = (\delta - J) + g = \delta . 
\end{equation}
We prefer to work with functions rather than distributions, and we only use $\delta$ as a notation for the convolution identity. 

The assumptions on $J, g : \Rd \to \R$ are as follows.
\begin{assumption} \label{ass:J}
For both $h = J$ and $h = g$, 
we assume that $h$ is an even function (\ie, $h(-x) = h(x)$ for all $x$), that
\begin{equation} \label{eq:J_hyp}
h(x) \in L^1 \cap L^2(\R^d), 	\qquad
\abs x^2 h(x) \in L^1 \cap L^2(\R^d), 
\end{equation}
and that
\begin{equation} \label{eq:J_2+eps}
\abs x^{2+\eps} h(x) \in L^1(\R^d) 
\qquad \text{ for some $\eps>0$.}
\end{equation} 
If $d>4$, we further assume that
\begin{equation} \label{eq:J_d-2}
\abs x^{d-2} h(x) \in L^{p} \cap L^2(\R^d)
\qquad \text{for some $1 \le p < \frac d {4}$}. 
\end{equation}
For $J$, we assume $\hat J(0) = 1$ and
the following \emph{infrared bound:}
there is a constant $K_{\rm IR} > 0$ such that 
\begin{equation} \label{eq:infrared}
\hat J(0) - \hat J(k) \ge \KIR (\abs k^2 \wedge 1)
\qquad (k\in \Rd).
\end{equation}
\end{assumption}

\begin{remark}
\begin{enumerate}[label=(\roman*)]
\item
Conditions \eqref{eq:J_hyp}--\eqref{eq:J_d-2}
are implied by the decay estimate
\begin{align} \label{eq:h_decay}
\abs{ h(x) } \le \frac C { (1 + \abs x)^{d+2+\rho} } 
\qquad (x\in \Rd)
\end{align}
where $\rho > \frac{ d-8 } 2 \vee 0$.
This decay assumption is used in \cite{LS24a}.

\item
We do not assume $J(x) \ge 0$, so we cannot interpret $\delta - J$ as the generator of a random walk. 
Allowing negative values of $J(x)$ is important for applications based on the lace expansion. 

\end{enumerate}
\end{remark}

To handle the non-compact Fourier dual of $\Rd$, 
it is more convenient to work with the function $H = G - g$. Note that \eqref{eq:JG} holds if and only if
\begin{equation} \label{eq:JH}
(\delta - J) * H = J * g .
\end{equation}
Under Assumption~\ref{ass:J}, a solution to \eqref{eq:JH} is given by the Fourier integral
\begin{equation} \label{eq:Hint}
H(x) = \int_\Rd \frac{ \hat J(k) \hat g(k) }{ 1 - \hat J(k) }  e^{-ik\cdot x} \frac{ dk }{ (2\pi)^d } .
\end{equation}
The integral is well-defined in dimensions $d>2$,
because the integrand is bounded by $\KIR \inv \abs k^{-2} \norm J_1 \norm g_1$ when $\abs k \le 1$,
and it is
bounded by $\KIR\inv \abs{ \hat J(k) } \abs{ \hat g(k) }$ when $\abs k \ge 1$, which is integrable by the Cauchy--Schwarz inequality.
Also, if $d>4$, then \eqref{eq:Hint} is the unique solution to \eqref{eq:JH} in $L^2(\Rd)$, by the $L^2$ Fourier transform.

\begin{theorem}[Gaussian deconvolution]
\label{thm:deconv}
Let $d > 2$, let $a_d = \frac{ \Gamma(\frac{ d-2 } 2 ) }{ 2\pi^{d/2}}$, 
and let $J,g$ obey Assumption~\ref{ass:J}. 
Then the Fourier integral solution $H(x)$ to \eqref{eq:JH}, given in \eqref{eq:Hint}, obeys
\begin{align} \label{eq:H_asymp}
H(x) = 
\frac{ a_d \int_\Rd g(y) dy }{ \sqrt{\det \Sigma} } \frac 1 { ( x \cdot \Sigma\inv x )^{(d-2)/2} } 
	+ o\bigg( \frac 1 {\abs x^{d-2}} \bigg)
	\qquad \text{as $\abs x \to \infty$},
\end{align}
where $\Sigma$ is a diagonal matrix $\Sigma = \mathrm{diag} ( \int_\Rd x_i^2 J(x) dx : 1 \le i \le d )$. 

Additionally, if $g(x) = o(\abs x^{-(d-2)})$ as $\abs x\to \infty$, 
then the solution $G(x) = H(x) + g(x)$ to \eqref{eq:JG} obeys the same asymptotics \eqref{eq:H_asymp}. 
\end{theorem}

We remark that the condition $\hat J(0) = 1$ in Assumption~\ref{ass:J} is a criticality condition and is responsible for the polynomial decay of $H(x)$ in \eqref{eq:H_asymp}.
If $\hat J(0) < 1$ instead, we expect $H(x)$ to decay exponentially. 
With ideas from \cite{Liu24, Slad23_wsaw}, we expect the methods developed in this paper to extend to subcritical two-point functions, to yield a uniform upper bound
\begin{equation} \label{eq:near_crit}
H(x) \le \frac C { 1 \vee \abs x^{d-2} } e^{- c m_J \abs x} 
	\qquad (x\in \Rd) ,
\end{equation}
where $c\in (0,1)$ and $m_J$ is the rate of exponential decay of $H(x)$.
We expect \eqref{eq:near_crit} to be useful in studying statistical mechanical models on the continuum torus; see \cite{LPS25t_prep} and references therein for the discrete setting.

\subsection{Application to self-repellent Brownian motion}

We apply our deconvolution theorem to a self-repellent Brownian motion in dimensions $d > 4$, studied recently by \cite{BKM24}. 
This model is similar to the weakly self-avoiding walk on $\Zd$, but it allows a penalty based on path interactions. 

The model is defined as follows. For $N \in \N$, 
let $\Ccal_N$ denote the set of continuous functions from $[0,N]$ to $\Rd$.
For a function $B \in \Ccal_N$ and an integer $j \in [1, N]$, 
we define the \emph{$j$-th leg} of $B$, denoted by $B_j$, to be the function $B_j(s) = B( j-1 + s)$, defined for $s \in [0,1]$. 
We fix a bounded continuous function $v : [0,\infty) \to [0,\infty)$ with compact support, and define the following Hamiltonian for $B \in \Ccal_N$:
\begin{align}
H_N(B) = \sum_{1 \le i < j \le N } V( B_i, B_j ) ,
\qquad
V(f,g) = \int_0^1 v(\abs{f(s) - g(s)}) ds .
\end{align}
Let $\alpha > 0 $ be a parameter, 
and let $\P_N$ denote the standard Wiener measure on $\Ccal_N$. 
We define the measure $\Q_{\alpha, N}$ on $\Ccal_N$ by
\begin{equation}
\frac{d \Q_{\alpha, N} }{ d \P_N} (B) = e^{-\alpha H_N(B)} .
\end{equation}
The normalised version of $\Q_{\alpha, N}$ is the self-repellent Brownian motion. 
The \emph{two-point function} for the self-repellent Brownian motion is defined, for $\lambda \ge 0$, as
\begin{align}
G_{\alpha,\lambda}(x) = \sum_{N=1}^\infty \lambda^N \Gamma_{\alpha,N}(x), 
\qquad \text{where} \quad
\Gamma_{\alpha,N}(x) = \frac{ \Q_{\alpha,N}( B_N \in dx ) } { dx }
\end{align}
is the density function of the marginal distribution of $B_N$.
A standard subadditivity argument implies the existence of $\lambda_c(\alpha) $ such that 
$ \norm{ G_{\alpha,\lambda} }_1 < \infty$ if and only if $\lambda <  \lambda_c(\alpha)$, so the sum defining $G_{\alpha,\lambda}(x)$ converges at least for $\lambda <  \lambda_c(\alpha)$.

In dimensions $d > 4$,
using a new lace expansion, 
\cite[Theorem~4.1]{BKM24} established the following Gaussian domination bound for $\alpha$ sufficiently small: For all $x \in \Rd$,
\begin{align} \label{eq:BM_dominate}
G_{\alpha, \lambda_c(\alpha) } (x)
\le 5 C_\varphi(x) ,
\qquad
C_\varphi(x) = \sum_{n=1}^\infty \varphi_n(x) ,
\end{align}
where $\varphi_t(x) = (2\pi t)^{-d/2} \exp( - \abs x^2 / 2t )$ is the density function of the (centred) Gaussian distribution on $\R^d$ with covariance matrix $t \times \mathrm{Id}$. 
In particular, this implies 
$G_{\alpha, \lambda_c(\alpha) } (x) \le O(\abs x^{-(d-2)})$ as $\abs x \to \infty$, since $C_\varphi(x) \sim a_d \abs x^{-(d-2)}$ (see Lemma~\ref{lem:C} or \cite[Lemma~7.3]{BKM24}).
Building on their results, we improve to an asymptotic formula. 

\begin{theorem} \label{thm:BM}
Let $d > 4$ and $\alpha$ be sufficiently small.
Then there is a constant $c_d = a_d (1 + O(\alpha))$ such that 
\begin{equation} \label{eq:Gcrit_asymp}
G_{\alpha, \lambda_c(\alpha) } (x) 
\sim \frac{ c_d } { \abs x^{d-2} }  
\qquad \text{as $\abs x \to \infty$} .
\end{equation}
\end{theorem}

The proof of Theorem~\ref{thm:BM} is a direct verification of the hypotheses of Theorem~\ref{thm:deconv},
using results obtained in \cite{BKM24}.
Besides giving an asymptotic formula, 
we believe our methods can also be used to give an alternative bootstrap argument, similar to that for the weakly self-avoiding walk in \cite{Slad22_lace}, 
which produces Theorem~\ref{thm:BM} directly. 
We do not pursue this here.

\subsection{Strategy of proof}

For the proof of Theorem~\ref{thm:deconv},
we follow the framework of \cite{LS24a, Slad22_lace} to isolate the leading decay of $H(x)$ using a random walk two-point function, and then we show that the remainder is an error term using Fourier analysis. 
The idea of isolating the leading decay originated from \cite{HHS03}.
Unlike all previous works, we choose a random walk based on the function $J(x)$.

We use  a Gaussian random walk on $\Rd$.
Given a positive-definite matrix $\Sigma \in \R^{d \times d}$ (to be chosen later), 
we denote the density function of the (centred) Gaussian distribution on $\R^d$ with covariance matrix $\Sigma$ by
\begin{equation}
D(x) = 
\frac 1 {(2\pi)^{d/2} \sqrt{ \det \Sigma} } \exp( -x \cdot \Sigma\inv x / 2 ) .
\end{equation} 
Using the fact that the Fourier transform of a Gaussian function is Gaussian, we readily get an infrared bound 
\begin{equation}\label{eq:D_infrared}
1 - \hat D(k)  = \hat D(0) - \hat D(k) 
\ge K_{ {\rm IR}, \Sigma} (\abs k^2 \wedge 1)
	\qquad (k\in \R^d)
\end{equation}
with some constant $K_{ {\rm IR}, \Sigma} > 0$. 
With $D^{*n}$ denoting the $n$-fold convolution of $D$ with itself,
we define
\begin{equation}
C (x) = \sum_{n=2}^\infty D^{*n}(x),
\end{equation} 
which is the critical two-point function of the random walk without the zeroth and the first step. 
The function $C(x)$ satisfies the recurrence relation $C =  D^{*2} + D * C$, and it admits the Fourier integral representation
\begin{equation} \label{eq:Cint}
C(x) = \int_{\R^d} \frac{  \hat D(k)^2 }{ 1 -  \hat D(k) }  e^{-ik\cdot x} \frac{ dk }{ (2\pi)^d }
\end{equation}
in dimensions $d>2$ (cf. \eqref{eq:JH}--\eqref{eq:Hint}).
We will use the decay of $C(x)$, given by the next lemma. 

\begin{lemma} \label{lem:C}
Let $d > 2$ and let $ a_d = \frac{ \Gamma(\frac{ d-2 } 2 ) }{ 2\pi^{d/2}}$. 
We have
\begin{equation}
C(x)  =  \frac{ a_d }{ \sqrt{\det \Sigma} } \frac 1 { ( x \cdot \Sigma\inv x )^{(d-2)/2} } 
	+ O_\Sigma \bigg( \frac 1 {  \abs x^{d+2} } \bigg)
	\qquad \text{as $\abs x \to \infty$.}
\end{equation}
\end{lemma}

\begin{proof}
Since $D(x)$ is Gaussian, the convolutions $D^{*n}(x)$ can be calculated explicitly. By adding the first step of the random walk and then computing the series using \cite[Lemma~4.3.2]{LL10}, 
we have
\begin{align}
D(x) + C(x)
&= \frac 1 { (2\pi )^{d/2} \sqrt{\det \Sigma} } 
	\sum_{n=1}^\infty \frac 1 { n ^{d/2} } e^{ - x \cdot \Sigma\inv x / 2n } \nl
&= \frac{ \Gamma(\frac{ d-2 } 2 ) }{ 2\pi^{d/2} \sqrt{\det \Sigma} } 
	\frac 1 { ( x \cdot \Sigma\inv x )^{(d-2)/2} }
	+ O \bigg( \frac 1 {  ( x \cdot \Sigma\inv x )^{(d+2)/2} } \bigg)
\end{align}
as $\abs x \to\infty$. 
This gives the desired result since $D(x)$ can be absorbed into the error term. 
\end{proof}

Let $J,g$ obey Assumption~\ref{ass:J}. 
In view of the decay of $C(x)$, 
we want to choose the matrix $\Sigma$ appropriately, so that the decomposition
\begin{equation} \label{eq:decomp}
H(x) = \hat g(0) C(x) + f(x)
\end{equation}
gives a remainder $f(x)$ that decays faster than $\abs x^{-(d-2)}$. 
Since both $H(x)$ and $C(x)$ can be written as Fourier integrals, we decompose in the Fourier space, as
\begin{align} \label{eq:decomp_fourier}
\hat H
= \frac{ \hat J \hat g }{ 1 - \hat J } 
= \hat g(0) \frac{ \hat D^2 }{ 1 - \hat D }  
	+ \frac{ \hat g \hat J (1-\hat D) - \hat g(0) \hat D^2(1-\hat J) }{ (1-\hat D)(1-\hat J) }
= \hat g(0) \hat C + \hat f, 
\end{align}
where $\hat f$ is \emph{defined by}
\begin{equation} \label{eq:def_E}
\hat f =  \frac{ \hat E }{ (1-\hat D)(1-\hat J) }, 
\qquad
E = (g*J - g*J*D) - \hat g(0) (D*D - D*D*J). 
\end{equation}
The remainder $f(x)$ in \eqref{eq:decomp} is then given by the inverse Fourier transform of $\hat f$.

The good choice of $\Sigma$ turns out to be the diagonal matrix
\begin{align} \label{eq:def_Sigma}
\Sigma = \mathrm{diag} \bigg( \int_\Rd x_i^2 J(x) dx : 1 \le i \le d \bigg) , 
\end{align}
which is positive-definite by the evenness of $J(x)$ and
the infrared bound \eqref{eq:infrared}. 
The choice \eqref{eq:def_Sigma} is to ensure that all second derivatives of $\hat E(k)$ vanish at $k=0$. 
Intuitively, since also $\hat E(0) = 0$ and all first derivatives of $\hat E(k)$ vanish at $k=0$ by evenness of $E(x)$, we expect $\abs{\hat E(k)} \approx \abs k^{2+\sigma}$ for some $\sigma > 0$. Combined with the infrared bounds \eqref{eq:infrared} and \eqref{eq:D_infrared}, we get that $\abs{ \hat f(k) } \lesssim \abs k^{\sigma - 2}$, which is less singular than $\abs{\hat C(k)} \approx \abs k^{-2}$ near $k=0$. This more regular behaviour should transfer to faster decay of $f(x)$ as $\abs x\to \infty$. 
The next proposition gives a precise statement about the regularity of $\hat f$. 

We use the notation of \emph{weak derivatives}; see \cite[Chapter~5]{Evan10} or \cite[Appendix~A]{LS24a} for an introduction. 
(\cite{LS24a} works with the continuum torus $\T^d = (\R / 2\pi \Z)^d$, but proofs for $\Rd$ are the essentially same and require only replacing all $L^p(\T^d)$ spaces by \emph{local} $L^p(\Rd)$ spaces, since test functions are compactly supported.)

\begin{proposition} \label{prop:f}
Let $d > 2$
and let $J,g$ obey Assumption~\ref{ass:J}. 
Define $\Sigma$ by \eqref{eq:def_Sigma}.
Then $\hat f$ is $d-2$ times weakly differentiable on $\R^d$, and $\grad^\alpha \hat f \in L^1(\R^d)$ for all multi-indices $\alpha$ with $\abs \alpha \le d-2$. 
\end{proposition}

\begin{proof}[Proof of Theorem~\ref{thm:deconv} assuming Proposition~\ref{prop:f}]
By the decomposition \eqref{eq:decomp} and the decay of $C(x)$ given in Lemma~\ref{lem:C}, 
it suffices to prove that $f(x) = o(\abs x^{-(d-2)})$ as $\abs x \to \infty$. 
By Proposition~\ref{prop:f},
$\grad^\alpha \hat f \in L^1(\Rd)$
for all multi-index $\alpha$ with $\abs \alpha = d-2$.
Since the inverse Fourier transform of $\grad^\alpha \hat f$
is a constant multiple of $x^\alpha f(x)$, the Riemann--Lebesgue lemma implies that $x^{\alpha} f(x) \to 0 $ as $\abs x \to \infty$. 
But this holds for all $\abs \alpha = d-2$, so we get  $ \abs x^{d-2} f(x) \to 0 $, which concludes the proof. 
\end{proof}

The rest of the paper is organised as follows.
In Section~\ref{sec:regularity}, 
we prove elementary regularity estimates using basic Fourier analysis and H\"older's inequality. 
In Section~\ref{sec:pf}, 
we apply these estimates to prove Proposition~\ref{prop:f}, 
following the strategy of \cite{LS24a}.
This finishes the proof of Theorem~\ref{thm:deconv}.
In Section~\ref{sec:BM}, we prove Theorem~\ref{thm:BM} using Theorem~\ref{thm:deconv} and results of \cite{BKM24}.

\section{Regularity estimates}
\label{sec:regularity}

We prove three lemmas in this section. They will be applied to $h=J,g,D, \text{or }E$. 
The first lemma gives a way to interpolate moments of a function. 
It could also be used to verify equation \eqref{eq:J_d-2} of Assumption~\ref{ass:J}, by choosing $b = d-2$.

\begin{lemma}
\label{lem:interp}
Let $d\ge 1$.
Let $h : \R^d \to \C$ be a function such that
\begin{equation} \label{eq:ass_h}
h(x) \in L^1 \cap L^2(\R^d), 
\qquad
\abs x^{2} h(x) \in L^1(\R^d) .
\end{equation}
Let $ 2 \le b \le a \le d+2$ and write $p_a^* = \frac d {d - a + 2} \in [1,\infty]$. Then
\begin{equation} \label{eq:h_hyp}
\abs x^a h(x) \in L^{p_a} \cap L^2(\R^d)
\qquad
\text{for some $1 \le p_a \le p_a^*$} 
\end{equation}
implies
\begin{equation} \label{eq:h_conc}
\abs x^b h(x) \in L^{p_b} \cap L^2(\R^d)
\qquad
\text{for some $1 \le p_b \le p_b^*$} . 
\end{equation}
Moreover, if $b > 2$ and \eqref{eq:h_hyp} holds with some $p_a < p_a^*$, then \eqref{eq:h_conc} holds with some $p_b < p_b^*$. 
\end{lemma}

\begin{proof}
Let $b \in [2, a]$. 
Since $h(x)$ and $\abs x^a h(x)$ are both in $L^2$ by the hypotheses, we get $\abs x^b h(x) \in L^2$.
Also, using
$\abs x^2 h(x) \in L^1$, $\abs x^a h(x) \in L^{p_a}$,
and the decomposition 
\begin{equation}
\abs x^b h(x) =  \big( \abs x^{ 2 } h(x) \big)^{ \frac{a-b}{a-2} } 
	\big( \abs x^{ a } h(x) \big)^{ \frac{b-2}{a-2} } ,
\end{equation}
it follows from H\"older's inequality that
\begin{align} \label{eq:xbh}
\abs x^b h(x) \in L^{p_b}(\R^d) 
\quad \text{with} \quad
\frac 1 {p_b} = \frac{ a-b }{a-2} + \frac { b-2 }{ (a-2) p_a } .
\end{align}
A short computation 
using $(p_a^*)\inv = 1 - \frac{a-2}d$
then shows 
that $p_b \le p_b^*$ when $p_a \le p_a^*$, 
and that $p_b < p_b^*$ when $b > 2$ and $p_a < p_a^*$. 
This concludes the proof. 
\end{proof}

The next two lemmas estimate regularity of the Fourier transform. 
Lemma~\ref{lem:h} handles local behaviour and Lemma~\ref{lem:prod_h} handles global behaviour.
We use local $L^p$ spaces. 
The space $L^p\loc(\Rd)$ consists of all measurable functions $u : \Rd \to \C$ such that $\left. u \right\rvert_K  \in L^p(K)$ for all compact subsets $K \subset \Rd$. 
These spaces are nested, in the sense that 
$L^p\loc \subset L^q\loc$ when $p \ge q$.

\begin{lemma} 
\label{lem:h}
Let $d\ge 1$, $a \in [2, d+2]$, and
$h: \R^d \to \C$ be a function such that \eqref{eq:ass_h}--\eqref{eq:h_hyp} hold. 
Then 
\begin{enumerate}[label=(\roman*)]
\item
The Fourier transform $\hat h$
is $\floor a$ times weakly differentiable on $\R^d$, 
and the derivatives satisfy
\begin{alignat}{2} 
\grad^\gamma \hat h &\in L^2 \cap L^\infty(\Rd)
	\qquad &&( 0 \le \abs \gamma \le 2 ), 
	\label{eq:h_gamma_2} \\
\grad^\gamma \hat h &\in L^2 \cap L^{ \frac d {\abs \gamma - 2} }\loc(\Rd)
	\qquad &&(3 \le \abs \gamma \le a ). 	
	\label{eq:h_gamma_3}
\end{alignat}
Moreover, if $ p_a  <  p_a^*$ in \eqref{eq:h_hyp}, there exists $\eps \in (0,1]$ such that
\begin{alignat}{2} 
\grad^\gamma \hat h &\in L^{ \frac d {\abs \gamma - 2 - \eps} }\loc(\Rd)
	\qquad~\quad &&(3 \le \abs \gamma \le a ) . 	
	\label{eq:h_gamma_3+}
\end{alignat}

\item
Additionally, 
if $h(x)$ is an even function and satisfies an infrared bound
\begin{align} 
\hat h(0) \le 1,
\qquad
\hat h(0) - \hat h(k) 
\ge K_{ {\rm IR}, h} (\abs k^2 \wedge 1)
	\qquad (k\in \R^d) ,
\end{align}
then for any multi-index $\gamma$ with $1 \le \abs \gamma \le a$, we have
\begin{align}
\frac{ \grad^\gamma \hat h }{ 1 - \hat h } \in L^q\loc(\R^d) 
	\qquad ( q\inv > \frac{ \abs \gamma } d ).
\end{align}

\end{enumerate}
\end{lemma}

\begin{proof}
(i) Let $\gamma$ be a multi-index
and let $b = \abs \gamma \le a$. 
Since $\abs x^b h(x) \in L^2$ by \eqref{eq:ass_h}--\eqref{eq:h_hyp},
the derivative $\grad^\gamma \hat h$ exists weakly (see \cite[Lemma~A.4]{LS24a} which also works on $\Rd$) and belongs to $L^2$. 

When $b \le 2$, 
\eqref{eq:ass_h} also implies $\abs x^b h(x) \in L^1$,
so $\grad^\gamma \hat h$ exists classically and belongs to $L^\infty$.

When $b > 2$, 
using the $L^p$ space for $\abs x^b h(x)$ computed in \eqref{eq:xbh},
using boundedness of the $L^p$ Fourier transform ($1 \le p \le 2$),
and using $p_a\inv \ge (p_a^*)\inv = 1 - \frac{a-2}d$,
we have
\begin{equation} \label{eq:h_gamma_proof}
\grad^\gamma \hat h \in L^2 \cap L^{q_b \vee 2}(\R^d) ,
\quad \text{where} \quad
\frac 1 {q_b} = 1 - \frac 1 {p_b}
	= \frac { b-2 }{ a-2 } \Big( 1 - \frac 1 {p_a} \Bigr)
	\le \frac{ b-2}d. 
\end{equation}
Since $L^q\loc$ spaces are nested, the claim \eqref{eq:h_gamma_3} follows.

Finally, if in \eqref{eq:h_hyp} we have $ p_a  < p_a^*$,
then the last inequality of \eqref{eq:h_gamma_proof} is strict,
so we can pick a small $\eps>0$ such that $\frac{ b - 2 - \eps }d > q_b\inv$ for all integer $b \in [3, a]$
to get \eqref{eq:h_gamma_3+}.

\smallskip \noindent 
(ii) 
Let $\B = \{ k \in \Rd : \abs k < 1 \}$. 
By the infrared bound, we have
\begin{align} \label{eq:h_loc_pf}
\biggabs{ \frac{ \grad^\gamma \hat h }{ 1 - \hat h } (k) }
\lesssim  \frac{ \abs{ \grad^\gamma \hat h (k)} }{ \abs k^2 } \1_\B(k) 
	+ \abs{ \grad^\gamma \hat h (k)} .
\end{align}
Since $L^q\loc$ spaces are nested, by part (i) the second term on the right-hand side is in $L^q\loc$ for all $q\inv \ge \abs \gamma / d$.
Therefore, we only need to estimate the first term. 

If $\abs \gamma = 1$, 
we use Taylor's theorem and evenness of $h(x)$ to bound 
$\abs{ \grad^\gamma \hat h(k) } \lesssim \abs k \bignorm{ \abs x^2 h(x) }_1 \lesssim \abs k$. 
Then the first term on the right-hand side of \eqref{eq:h_loc_pf} is bounded by a multiple of $\abs k \inv \1_\B$, which is in $L^q$ for all $q\inv > 1/d$, as desired.

If $\abs \gamma \in [2, a]$, 
we use H\"older's inequality,
that $\grad^\gamma \hat h \in L^{d/ (\abs \gamma - 2)}\loc$ from part (i),
and that $\abs k^{-2} \in L^p\loc$ for $p\inv > 2/d$, 
to get that their product is in $L^q\loc$ with $q\inv > (\abs \gamma - 2 + 2)/d = \abs \gamma / d$, as desired.
\end{proof}

\begin{lemma} \label{lem:prod_h}
Let $d\ge 1$, $a \in [2, d+2]$,
$n \ge 2$, and $\gamma_i$ be multi-indices such that $\sum_{i=1}^n \abs{ \gamma_i } \le a$. 
Let $h_i : \Rd \to \C$ be functions for each of which \eqref{eq:ass_h}--\eqref{eq:h_hyp} hold. 
Then
\begin{align} \label{eq:prod_h}
\prod_{i=1}^n \grad^{\gamma_i} \hat h_i    \in L^r(\R^d) 
\quad \text{for} \quad
\Big( \frac{ -2 + \sum_{i=1}^n \abs{ \gamma_i }  } d \Big) \vee 0
\le \frac 1 r \le 1 .
\end{align}
\end{lemma}

\begin{proof}
Since each $\grad^{\gamma_i} \hat h_i \in L^2$ by Lemma~\ref{lem:h}(i), the product $\prod_{i=1}^n \grad^{\gamma_i} \hat h_i$ belongs to $L^r$ where $r\inv = n/2 \ge 1$ by H\"older's inequality. 
We choose different $L^q$ spaces for the derivatives, to decrease $r\inv$ to the desired values. 
For the allowed range of $q$, 
we note that when $\abs{\gamma_i} \le 2$
we have $\grad^{\gamma_i} \hat h_i \in L^2 \cap L^\infty$. 
And when $a \ge 3$ and $\abs{ \gamma_i } \in [3, a]$, $\grad^{\gamma_i} \hat h_i$ belongs to the $L^q$ spaces identified in \eqref{eq:h_gamma_proof} with $b = \abs{ \gamma_i }$. 
Therefore, by taking larger $q$,
it is possible to decrease $r\inv$ from $n/2$ to $m$ where
\begin{align}
m \le 
\sum_{i: \abs{ \gamma_i } \le 2} \frac { 1 }{ \infty }  +
\sum_{i: \abs{ \gamma_i } \ge 3} \frac { \abs{ \gamma_i } - 2 }{ d }
\le \Big( \frac{ -2 + \sum_{i=1}^n \abs{ \gamma_i }  } d \Big) \vee 0 .
\end{align}
In particular, we can get the desired values of $r\inv$. 
\end{proof}

\section{Proof of Proposition~\ref{prop:f}}
\label{sec:pf}

We now prove Proposition~\ref{prop:f} following the strategy of \cite{LS24a}.
We want to estimate $d-2$ derivatives of the function
$\hat f = \frac{ \hat E }{ (1 - \hat D) ( 1 - \hat J) }$ defined in \eqref{eq:def_E}. 
Let $\alpha$ be a multi-index with $\abs \alpha \le d-2$.
Using the product and quotient rules of weak derivatives \cite[Appendix~A]{LS24a}, 
to estimate $\grad^\alpha \hat f$
it suffices to estimate all terms of the form
\begin{equation} \label{eq:f_alpha_decomp_0}
\left( \prod_{l=1}^m \frac{  - \grad^{\delta_ l} \hat D  }{ 1 - \hat D } \right)
\frac{ \grad^{\alpha_2} \hat E }{ (1-\hat D)(1-\hat J) }
\left( \prod_{i=1}^n \frac{  - \grad^{\gamma_ i} \hat J  }{ 1 - \hat J } \right),
\end{equation}
where 
$\alpha = \alpha_1 + \alpha_2 + \alpha_3$, 
$0 \le m \le \abs{\alpha_1}$, 
$0 \le n \le \abs{\alpha_3}$,
$\alpha_1 = \sum_{l=1}^m \delta_l $, 
$\alpha_3 = \sum_{i=1}^n \gamma_i$,
and $\abs{\delta_l} \ge 1$, $\abs {\gamma_i} \ge 1$.
Once we show that \eqref{eq:f_alpha_decomp_0} is locally integrable, 
it will follow that $\hat f$ is $d-2$ times weakly differentiable, 
and that $\grad^\alpha \hat f$ is given by a linear combination of \eqref{eq:f_alpha_decomp_0}. 
Thereafter, we show $\eqref{eq:f_alpha_decomp_0} \in L^1(\Rd)$ to conclude the proposition. 

Equation \eqref{eq:f_alpha_decomp_0} has been factored so that we can apply Lemma~\ref{lem:h}(ii), with $h=D\text{ or } J$, to the first and the last term. 
The middle term of \eqref{eq:f_alpha_decomp_0} is estimated by
the following lemma; it crucially relies on the choice of the matrix $\Sigma$ in \eqref{eq:def_Sigma}.

\begin{lemma} \label{lem:E}
Let $d > 2$
and let $J,g$ obey Assumption~\ref{ass:J}. 
Define $\Sigma$ by \eqref{eq:def_Sigma}.
Then there exists $\eps' \in (0, 1)$ such that
\begin{equation} \label{eq:EAF}
\frac{ \grad^{\gamma} \hat E }{ (1-\hat D)(1-\hat J) }
	\in L^{ \frac d {\abs \gamma + 2 - \eps'} }\loc(\Rd)
	\qquad (\abs \gamma \le d-2).
\end{equation}
\end{lemma}

\begin{proof}
Recall from \eqref{eq:def_E} that the function $E$ is defined as
\begin{equation} \label{eq:E_pf}
E = (g*J - g*J*D) - \hat g(0) (D*D - D*D*J). 
\end{equation}
We estimate $\grad^\gamma \hat E$ in two cases. 

When $3 \le \abs \gamma \le d-2$, which only happens in dimensions $d>4$, 
we use Lemma~\ref{lem:h}(i) with 
$a=d-2$ and $h = E$.
The required moment conditions on $E$ follow from Assumpton~\ref{ass:J} and Young's convolution inequality. 
In particular, 
since the inequality $p < \frac d 4$ is strict in \eqref{eq:J_d-2}, 
we obtain the improved estimate
\begin{equation} 
\grad^\gamma \hat E \in L^{ \frac d {\abs \gamma - 2 - \eps} }\loc(\Rd)
	\qquad (3 \le \abs \gamma \le d-2 ) 
\end{equation}
with some $\eps \in (0,1]$ from \eqref{eq:h_gamma_3+}.
Combined with 
\begin{equation}
\biggabs{ \frac{ 1 }{ (1-\hat D)(1-\hat J) } (k) }
\lesssim \frac 1 { \abs k^4 \wedge 1 }
\in L^{ p }\loc(\Rd) 
	\qquad ( p \inv > \frac 4 d )
\end{equation}
from the infrared bounds \eqref{eq:infrared} and \eqref{eq:D_infrared},
H\"older's inequality implies that
\begin{align}
\frac{ \grad^{\gamma} \hat E }{ (1-\hat D)(1-\hat J) }
	\in L^{ r }\loc(\Rd)
	\quad \text{for} \quad
	\frac 1 r > \frac{ \abs \gamma - 2 - \eps } d + \frac 4 d 
	= \frac{ \abs \gamma + 2 - \eps} d. 
\end{align}
This gives the desired result with $\eps' \in (0,\eps')$.

When $\abs \gamma \le 2$, 
we expand $\grad^\gamma \hat E(k)$ using the choice of $\Sigma$ in \eqref{eq:def_Sigma}.
We first observe that
\begin{equation}
\int_\Rd E(x) dx = \hat E(0) = \hat g(0) \hat J(0) [1-\hat D(0)] - \hat g(0) \hat D(0)^2 [1-\hat J(0)] = 0
\end{equation}
since $\hat J(0) = \hat D(0) = 1$,
that $\int_\Rd x_i x_j E(x) dx = 0$ for all $i\ne j$ by evenness of $E(x)$, 
and that
\begin{equation}
\begin{aligned}
-\int_\Rd x_i^2 E(x) dx
=\grad_{ii} \hat E(0) 
&= \hat g(0) \hat J(0) [ - \grad_{ii} \hat D(0) ] - \hat g(0) \hat D(0)^2 [ - \grad_{ii} \hat J(0) ] 	\\
&= \hat g(0) \bigg( \int_\Rd x_i^2 D(x) dx - \int_\Rd x_i^2 J(x) dx \bigg)
= 0
\end{aligned}
\end{equation}
for all $i$ by the choice of $\Sigma$. 
Then, with $h_x(k) = \cos(k\cdot x) - 1 + \frac{ (k\cdot x)^2 }{2!}$, 
we can write
\begin{equation}
\hat E(k) = \int_\Rd E(x) \cos(k\cdot x) dx = \int_\Rd E(x) h_x(k) dx.
\end{equation}
Since $\abs x^{2+\eps_2} E(x) \in L^1(\Rd)$ for some $\eps_2 \in (0,2]$
by \eqref{eq:J_2+eps} in Assumption~\ref{ass:J}
and Young's convolution inequality, 
by expanding $\abs{ \grad^\gamma h_x(k) } \lesssim \abs k^{2+\eps_2 - \abs \gamma} \abs x^{2+\eps_2}$,
we have
\begin{align}
\abs{ \grad^\gamma \hat E (k) }
\le \int_\Rd \abs {E(x)} \abs{ \grad^\gamma h_x(k) } dx
\lesssim  \abs k^{2+\eps_2 - \abs \gamma} \int_\Rd  \abs x^{2+\eps_2} \abs{E(x)} dx 
\lesssim \abs k^{2+\eps_2 - \abs \gamma} .
\end{align}
Combined with the infrared bounds \eqref{eq:infrared} and \eqref{eq:D_infrared}, we finally get
\begin{align} 
\biggabs{ \frac{ \grad^{\gamma} \hat E }{ (1-\hat D)(1-\hat J) } (k)  }
&\lesssim  \frac{ \abs k^{2 + \eps_2 - \abs \gamma} }{ \abs k^2 \abs k ^2 } \1_\B
	+ \norm{ \grad^\gamma \hat E}_\infty 
=  \frac{ 1 }{ \abs k^{\abs \gamma + 2 - \eps_2 } } \1_\B
	+ \norm{ \grad^\gamma \hat E }_\infty  ,
\end{align}
which is in $L^r\loc(\Rd)$ when $r\inv > ( \abs \gamma + 2 - \eps_2 )/d$.
Taking $\eps' \in (0, \eps_2)$ then gives the desired result. 
\end{proof}

\begin{proof}[Proof of Proposition~\ref{prop:f}]
As discussed in the paragraph containing \eqref{eq:f_alpha_decomp_0}, our goal is to prove that $\eqref{eq:f_alpha_decomp_0} \in L^1(\Rd)$. 
Under Assumption~\ref{ass:J},
we can apply all lemmas in Section~\ref{sec:regularity} with $a = (d-2) \vee 2$ to $h = J \text{ or } g$.
Since $D(x)$ decays exponentially and satisfies the infrared bound \eqref{eq:D_infrared}, the lemmas apply to $h=D$ too.

We first prove that \eqref{eq:f_alpha_decomp_0} is locally integrable.
Since \eqref{eq:f_alpha_decomp_0} is a product, 
we use H\"older's inequality, 
with Lemma~\ref{lem:h}(ii) on the first and the last factor, 
and with Lemma~\ref{lem:E} on the middle factor,
to get that $\eqref{eq:f_alpha_decomp_0} \in L^r\loc(\Rd)$ for
\begin{equation}
\frac 1 r > \frac{ \sum_{l=1}^m \abs{ \delta_l } } d
	+ \frac{ \abs{ \alpha_2 } + 2 - \eps' } d
	+ \frac{ \sum_{i=1}^n \abs{ \gamma_i } } d
= \frac { \abs \alpha + 2 - \eps' } d.
\end{equation}
Since $\abs \alpha \le d-2$ and $\eps'>0$, 
we can take $r = 1$ to get that \eqref{eq:f_alpha_decomp_0} is locally integrable.

To improve to $L^1(\Rd)$, 
it suffices to prove that \eqref{eq:f_alpha_decomp_0} is integrable on $\{ \abs k \ge 1 \}$.
On this domain, we can use the infrared bounds \eqref{eq:infrared} and \eqref{eq:D_infrared} to bound all denominators of \eqref{eq:f_alpha_decomp_0} by constants.
Then it suffices to show that
\begin{equation} \label{eq:f_alpha_decomp_2}
\left( \prod_{l=1}^m {  \grad^{\delta_ l} \hat D  } \right)
( \grad^{\alpha_2} \hat E )
\left( \prod_{i=1}^n {  \grad^{\gamma_ i} \hat J  } \right)
\end{equation}
is integrable over $\{ \abs k \ge 1 \}$.
We use Lemma~\ref{lem:prod_h} for this.
Since $\hat E = \hat g \hat J (1-\hat D) - \hat g(0) \hat D^2(1-\hat J)$ by the definition of $E$ in \eqref{eq:def_E}, 
once $\grad^{\alpha_2} \hat E$ is expanded using the product rule, we always have at least two factors of derivatives of $\hat g$, $\hat J$, or $\hat D$ in the product \eqref{eq:f_alpha_decomp_2}.
This allows the use of Lemma~\ref{lem:prod_h}.
As 
\begin{equation}
{-2 }
+ \( \sum_{l=1}^m \abs{\delta_ l} \)
+ \abs {\alpha_2} 
+ \( \sum_{i=1}^n \abs{\gamma_ i} \)
= -2 + \abs \alpha 
\le d-4, 
\end{equation}
Lemma~\ref{lem:prod_h} implies that $\eqref{eq:f_alpha_decomp_2} \in L^r( \Rd)$ for all $ (1 - \frac 4 d) \vee 0 \le r\inv \le 1$, 
which includes the desired $r=1$.  
This completes the proof. 
\end{proof}

\section{Self-repellent Brownian motion: proof of Theorem~\ref{thm:BM}}
\label{sec:BM}

In this section we prove Theorem~\ref{thm:BM}, 
which improves the Gaussian domination bound \eqref{eq:BM_dominate} obtained in \cite{BKM24} to an asymptotic formula. 
The proof is a straightforward verification of the hypotheses of Theorem~\ref{thm:deconv} using results of \cite{BKM24}. 
For simplicity, we omit $\alpha$ in our notations and write $G_\lambda = G_{\alpha, \lambda}$, $\lambda_c = \lambda_c(\alpha)$. 

We begin with the lace expansion equation satisfied by $G_\lambda$, which will be rearranged to \eqref{eq:JG}. 
By \cite[Remark~2]{BKM24} and the Gaussian domination bound \eqref{eq:BM_dominate},
for all $\lambda \le \lambda_c$, 
$G_\lambda$ satisfies the convolution equation
\begin{equation} \label{eq:Glam}
G_\lambda = (\lambda \varphi_1 + \Pi_\lambda) + (\lambda \varphi_1 + \Pi_\lambda) * G_\lambda ,
\end{equation}
where 
$\varphi_1(x) = (2\pi)^{-d/2} \exp( - \abs x^2 / 2 )$
and 
$\Pi_\lambda: \Rd \to \R$ is an explicit power series in $\lambda$ (denoted by $G_\lambda^{(\Pi)}$ in \cite{BKM24}). 
By \cite[Lemma~4.4]{BKM24} (using \eqref{eq:BM_dominate} to verify the hypothesis), $\Pi_\lambda$ obeys the uniform bound
\begin{equation} \label{eq:Pi_decay}
\abs{ \Pi_\lambda(x) } \le \frac{ O (\alpha) }{ ( 1 + \abs x )^{3(d-2)} }
	\qquad (x\in \Rd,\ \lambda \le \lambda_c) .
\end{equation}
This implies $\hat \Pi_\lambda(0) \to \hat \Pi\crit(0)$ as $\lambda \to \lambda_c^-$, by the Dominated Convergence Theorem. 

\begin{proof}[Proof of Theorem~\ref{thm:BM}]
Let $J = g = \lambda_c \varphi_1 + \Pi\crit$,
so that \eqref{eq:Glam} rearranges into \eqref{eq:JG}.
To use Theorem~\ref{thm:deconv},
we need to verify Assumption~\ref{ass:J} and 
to show that $g(x) = o(\abs x^{-(d-2)})$. 

We begin with the decay.
From \eqref{eq:Pi_decay} and the exponential decay of $\varphi_1$,
we immediately get 
\begin{equation}
\abs{ J(x) } = \abs{ g(x) }
\lesssim \frac 1 { (1 + \abs x)^{3(d-2)} } 
\qquad (x\in \Rd) .
\end{equation}
This also verifies \eqref{eq:h_decay} with $\rho = 2(d-4)$, which is strictly larger than $ \frac{d-8}2 \vee 0$ in all dimensions $d > 4$, so it implies \eqref{eq:J_hyp}--\eqref{eq:J_d-2} in Assumption~\ref{ass:J}.
For $\hat J(0) = 1$, 
we first consider $\lambda < \lambda_c$ and take the Fourier transform of \eqref{eq:Glam}. Since $\hat G_\lambda(0) = \norm{ G_\lambda }_1 \to \infty$ as $\lambda \to \lambda_c^-$ by the Monotone Convergence Theorem, by sending $\lambda \to \lambda_c^-$ in the transformed equation we get
\begin{equation} \label{eq:J_hat_0}
\hat J(0)
= \lambda_c \hat \varphi_1(0) + \hat \Pi\crit(0)
= \lim_{\lambda \to \lambda_c^-} 
	\frac{ \hat G_\lambda(0) } { 1 + \hat G_\lambda(0) }
=1. 
\end{equation}
Finally, since
\begin{equation}
\hat J(0) - \hat J(k)
= \lambda_c [ \hat \varphi_1(0) - \hat \varphi_1(k) ] 
	+ [ \hat \Pi\crit (0) - \hat \Pi\crit(k) ] ,
\end{equation}
and since
\begin{equation}
\abs{ \hat \Pi\crit (0) - \hat \Pi\crit(k) } 
\le O(\alpha) (\abs k^2 \wedge 1)
	\qquad (k \in \R^d)
\end{equation}
by \eqref{eq:Pi_decay} and Taylor's theorem,
the infrared bound for $J$ follows from that for $\varphi_1$ (obtained via a computation using the Gaussian function $\hat \varphi_1(k)$), by picking $\alpha$ small enough. 
This verifies all of Assumption~\ref{ass:J}. 

Since $d > 4$ and since $G\crit - g \in L^2(\Rd)$ by the Gaussian domination bound \eqref{eq:BM_dominate},
we know $G\crit - g$ is equal to the Fourier integral $H$ defined in \eqref{eq:Hint}, by the uniqueness of $L^2$ solutions.
Therefore, we can use Theorem~\ref{thm:deconv} to get the asymptotics of $G\crit = G$. 
Since $\int_\Rd g(y) dy = \hat g(0) = \hat J(0) = 1$ by \eqref{eq:J_hat_0}, 
and since $\Sigma = \sigma^2 \times \mathrm{Id}$ where $\sigma^2 = \int_\Rd x_1^2 J(x) dx$, we get
\begin{equation}
G\crit(x) \sim H(x) \sim  \frac{ a_d }{ \sigma^2}\frac 1{ \abs x^{d-2} }
	\qquad \text{as $\abs x \to \infty$}. 
\end{equation}
Also, since $\lambda_c = 1 - \hat \Pi\crit (0) = 1 - O(\alpha)$ by \eqref{eq:J_hat_0} and \eqref{eq:Pi_decay}, we have
\begin{equation}
\sigma^2
= \lambda_c \int_\Rd x_1^2 \varphi_1(x) dx + \int_\Rd x_1^2 \Pi\crit(x) dx
= [1 + O(\alpha)]( 1) + O(\alpha)
= 1 + O(\alpha) .
\end{equation}
This concludes the proof of Theorem~\ref{thm:BM}. 
\end{proof}

\section*{Acknowledgements}
The work was supported in part by NSERC of Canada.
We thank Matthew Dickson for discussions and for comments on a preliminary version.
We thank Gordon Slade for comments on a preliminary version.

\end{document}